\documentclass[12pt]{amsart}   % For Latex2e
\usepackage[english]{babel}
\usepackage{mathtools}
\usepackage{amssymb,amscd, graphics,color,latexsym,cancel}   % For Latex2e
\usepackage[linktoc=all, pagebackref, hyperindex]{hyperref}%
\hypersetup{colorlinks,  citecolor=blue,  filecolor=blue,  linkcolor=blue,  urlcolor=black}
\usepackage{tikz-cd} 
\usepackage{extpfeil}

\usepackage{tikz}
\usetikzlibrary{matrix,calc}
\usepackage{mathrsfs}
\usepackage[all]{xy}
\usepackage{comment}
\usepackage{amsfonts}
\usepackage[normalem]{ulem}

\usepackage{euscript}
\usepackage{amsmath}
\usepackage{ifpdf}
\usepackage{cleveref}
\LARGE\textwidth=6in
\newcommand{\lar}{\longrightarrow}
\newcommand{\surjects}{\twoheadrightarrow}

\newtheorem{Theorem}{Theorem}[section]

\newtheorem{Corollary}[Theorem]{Corollary}
\newtheorem{Proposition}[Theorem]{Proposition}

\theoremstyle{definition}
\newtheorem{Remark}[Theorem]{Remark}
\newtheorem{Example}[Theorem]{Example}
\newtheorem{Conjecture}[Theorem]{Conjecture}

\newtheorem{Question}[Theorem]{Question}
\def\sqr#1#2{{\vcenter{\hrule height.#2pt
			\hbox{\vrule width.#2pt height#1pt \kern#1pt
				\vrule width.#2pt}
			\hrule height.#2pt}}}
\def\phi{\varphi}

\def\VaVa{{\mathcal V}\kern-5pt {\mathcal V}}
\def\gr#1#2{{\rm gr}\, _{#1}(#2)}
\def\gr{{\rm gr}\,}
\def\hht{{\rm ht}\,}
\def\depth{{\rm depth}\,}

\def\Min{{\rm Min}\,}
\def\codim{{\rm codim}\,}

\def\ker{{\rm ker}\,}
\def\grade{{\rm grade}\,}
\def\rk{\rm rank}

\def\sym#1#2{\mbox{\rm Sym}_{#1}(#2)}

\def\Ext#1#2#3#4{{\rm Ext}\,^{#1}_{#2}({#3},{#4})}

\def\supp#1{{\rm Supp}\, (#1)}

\def\ini{\mbox{\rm in}}

\def\Rees{{\mathcal R}}
\def\sym{{\mathrm{Sym}}}
\def\cl#1{{\mathcal #1}}

\def\Rees{\mathcal{R}}

\def\phi{\varphi}
\def\Rees{{\cal R}}
\def\hht{{\rm ht}\,}
\def\grade{{\rm grade}\,}

\def\fm{{\mathfrak m}}
\def\fn{{\mathfrak n}}

\def\fp{{\mathfrak p}}

\def\fm{{\mathfrak m}}
\def\fn{{\mathfrak n}}

\def\cl#1{{\cal #1}}
\def\rk{\rm rank}

\newcommand{\excise}[1]{}
\def\NZQ{\mathbb}               % the font for N,Z,Q,R,C

\def\AA{{\NZQ A}}

\def\G{{\mathcal G}}

\def\Q{{\mathcal Q}}
\def\opn#1#2{\def#1{\operatorname{#2}}} % to make operators
\opn\chara{char} \opn\length{\ell} \opn\pd{pd} \opn\rk{rk}
\opn\projdim{proj\,dim} \opn\injdim{inj\,dim} \opn\rank{rank}
\opn\depth{depth} \opn\grade{grade} \opn\height{height}
\opn\embdim{emb\,dim} \opn\codim{codim}

\opn\Tr{Tr} \opn\bigrank{big\,rank}
\opn\superheight{superheight}\opn\lcm{lcm}
\opn\trdeg{tr\,deg}%\emph{
	\opn\reg{reg} \opn\lreg{lreg} \opn\ini{in} \opn\lpd{lpd}
	\opn\size{size} \opn\sdepth{sdepth}
	\opn\link{link}\opn\fdepth{fdepth}\opn\lex{lex}
	\opn\tr{tr}
	\opn\type{type}
	\opn\div{div} \opn\Div{Div} \opn\cl{cl} \opn\Cl{Cl}
	\opn\Spec{Spec} \opn\Supp{Supp} \opn\supp{supp} \opn\Sing{Sing}
	\opn\Ass{Ass} \opn\Min{Min}\opn\Mon{Mon}
	\opn\Ho{H}
	\opn\Ann{Ann} \opn\Rad{Rad} \opn\Soc{Soc}
	\opn\Im{Im} \opn\Ker{Ker} \opn\Coker{Coker} \opn\Am{Am}
	\opn\Hom{Hom} \opn\Tor{Tor} \opn\Ext{Ext} \opn\End{End}
	\opn\Aut{Aut} \opn\id{id}
	
	\opn\nat{nat}
	\opn\pff{pf}%   \pf exists already
	\opn\Pf{Pf} \opn\GL{GL} \opn\SL{SL} \opn\mod{mod} \opn\ord{ord}
	\opn\Gin{Gin} \opn\Hilb{Hilb}\opn\sort{sort}
	\opn\PF{PF}\opn\Ap{Ap}\opn\HF{HF}\opn\indeg{indeg}
	\opn\aff{aff} \opn
	\con{conv} \opn\relint{relint} \opn\st{st}
	\opn\lk{lk} \opn\cn{cn} \opn\core{core} \opn\vol{vol}  \opn\inp{inp} \opn\nilpot{nilpot}
	\opn\link{link} \opn\star{star}\opn\lex{lex}\opn\set{set}
	\opn\width{wd}
	\opn\Fr{F}
	\opn\QF{QF}
	\opn\G{G}
	\opn\type{type}\opn\res{res}
	\opn\log{Log}
	\opn\gr{gr}
	\def\Rees{{\mathcal R}}
	\def\pot#1#2{#1[\kern-0.28ex[#2]\kern-0.28ex]}

	\opn\dirlim{\underrightarrow{\lim}}
	\opn\inivlim{\underleftarrow{\lim}}
\begin{document}
		%\begin{titlepage}
		
		\title[Burity, Simis and Toh\u aneanu Conjectures]{Closing two recent conjectures related to the Jacobian ideal of hyperplane arrangements} 
		
		\author{Abbas Nasrollah Nejad}
		\address{Department of Mathematics, Institute for Advanced Studies in Basic Sciences (IASBS), Zanjan 45137-66731, Iran}
		\email{abbasnn@iasbs.ac.ir}
		
		\author{Aron Simis} 
		\address{Departamento de Matemática, Universidade Federal da Pernambuco, Recife, Pernambuco, 50740-560, Brazil,}
		\email{aron.simis@ufpe.br}
		\thanks{The second author was partially supported by a grant from CNPq (Brazil) (304800/2024-4).}
		
		%%%%%%%%%%%%%%%%%%%%%%%%%%%%%%%%%%%%%%%%%%%%%%%%%%%%%
		%%%%%%%%%%%%%%%%%%%%Absract%%%%%%%%%%%%%%%%
	\begin{abstract}
This work is about two conjectures stated by Burity--Simis--Toh\u{a}neanu regarding the Jacobian ideal of the defining polynomial of a central arrangement of $m$ hyperplanes. One settles one of these conjectures referring to the
Jacobian ideal being a minimal reduction of the ideal of $(m-1)$-fold products.  The second conjecture claiming the linear type property of the Jacobian ideal is disproved in rank at least four, by means of an explicit counter-example. In the latter the corresponding Rees algebra admits a torsion defining equation which is a Pfaffian syzygetic obstruction in degree two.
One also relates this Pfaffian obstruction to circuits and codimension-two flats of the arrangement.
	\end{abstract}
		\maketitle
		%%%%%%%%%%%%%%%%%%%%%%%%%%%%%%%%%%%%%%%%%%%%%%%%%%%%%%
		%%%%%%%%%%%%%%%%%%%%Introduction%%%%%%%%%%%%%%%%
		\section{Introduction}
		Let $\mathcal A$ be a central hyperplane arrangement of rank $n$ over a field $k$, and let $R=k[x_1,\ldots,x_n]$. Say, $\mathcal A$ is defined by linear forms
		$\ell_1,\ldots,\ell_m\in R_1$, with defining polynomial
		$f:=\ell_1\cdots \ell_m$.
		Setting $f_{x_i}:=\partial f/\partial x_i$, let
		\[
		J_f:=(f_{x_1},\ldots,f_{x_n})
		\]
		stand for the associated Jacobian (gradient) ideal. The properties and algebraic invariants of $J_f$ reflect both the geometry and the combinatorics of the arrangement. A distinguished role is played by the Rees algebra $\Rees_R(J_f)$ of $J_f$, expressing the homogeneous coordinate ring of the blowup along the Jacobian scheme, and its close associate the symmetric algebra $\sym_R(J_f)$, defined by the equations coming from the first syzygy module of the Jacobian ideal. One says that $J_f$ is an ideal of \emph{linear type} provided the canonical surjective algebra homomorphism
		\[
		\sym_R(J_f)\rightarrow  \Rees_R(J_f)
		\]
		is injective as well, in which case the isomorphism means that the blowup is defined by ``linear'' equations.
		
		One assume throughout that $k$ has zero (or sufficiently high) characteristic.

Jacobian ideals of arrangements have been studied from several points of view. For generic arrangements, the structure of logarithmic forms and logarithmic derivations was described by Yuzvinsky \cite{Yuzvinsky} and Rose--Terao \cite{RoseTerao}. 
The blowup algebra associated to the ideal $I\subseteq R$ generated by the $(m-1)$-fold products of the defining linear forms was studied by Garrousian--Simis--Toh\u{a}neanu \cite{GST}. More recently, Burity, Simis and
		Toh\u{a}neanu \cite{BST-almost} studied the relationship between this ideal and the Jacobian ideal $J_f$. Assuming that $k$ is as above,  they posed the following two conjectures:
		
		 \begin{Conjecture}[{\cite[Conjecture 1.1]{BST-almost}}] \label{red_conj}
		 	%\label{thm:BST-reduction-conjecture}
		 	With the above notation, $J_f$ is a minimal reduction of $I$, with reduction number
		 	$\operatorname{red}_{J_f}(I)\leq n-1$.
		 \end{Conjecture}
		 
		 In \emph{loc. cit.} the authors prove this conjecture in the case of an almost (i.e., $(n-1)$)-generic arrangement - in particular, for arrangements of rank at most $3$.
		
	%Among other results, they proved that $J_f$ is of linear type for central arrangements of rank $\leq 3$ and proposed the following conjecture.

		\begin{Conjecture}[{\cite[Conjecture~2.12]{BST-almost}}]\label{conj:BST}
		Notation as above, suppose that $m\geq n$. If $\operatorname{char}(k)\nmid m$, then the Jacobian ideal $J_f$ is of linear type.
		\end{Conjecture}
		It has been proved in \cite{BST-almost} that $J_f$ is of linear type provided the arrangement has rank $\leq 3$. 
		In addition, the conjecture is true for generic arrangements, as later obtained in  \cite[Theorem 3.10]{LORS}.
		
		The main goal of this paper is to prove the first of these conjectures  affirmatively, and to give examples to show that the second conjecture is false in any dimension $\geq 4$.
		
		The proof of the first conjecture is based on a criterion established in the original work by Burity, Simis and
		Toh\u{a}neanu in \cite{BST-almost}.
		For disproving the second conjecture one relies on a particular example of a rank four arrangement $\mathcal A$, by resorting to properties of the minimal free resolution of the corresponding Jacobian ideal, some of which were long established by Simis, Ulrich and Vasconcelos, involving the role of a Pfaffian element.
		One adds a discussion as to
		how the rank four counterexample is reflected in the combinatorics of the arrangement. The relevant features are the excess $m-n$, measuring the space of linear relations among the defining forms, and the
	    presence of $3$-circuits --  equivalently, triple codimension two flats. In the example, these data provide the combinatorial background for the quadratic Pfaffian equation appearing in the defining ideal of the Rees algebra of the Jacobian ideal.

		%%%%%%%%%%%%%%%%%%%%%%%%%%%%%%%%%%%%%%%%%%%%%%%
		\section{The minimal reduction conjecture}
		\label{sec:BST-reduction-conjecture}
		
		The present argument to prove Conjecture~\ref{red_conj} is based on the criterion established in \cite[Lemma~1.2]{BST-almost} in terms of the Orlik-Terao algebra. 
		
		We briefly expand on the main related players.
		
Let	$\mathcal A=\{H_i=V(\ell_i)\}_{i=1}^m$ be a central arrangement of rank $n$, with defining forms $\{\ell_1,\ldots,\ell_m\}\subseteq R:=k[x_1,\ldots,x_n]$.
With the notation established in the introduction, set in addition 
\[
L_i:=f/\ell_i,\qquad
I:=(L_1,\ldots,L_m).
\]
		Thus, $I=(L_1,\ldots,L_m)\subseteq R$ is the ideal of $(m-1)$-fold products of the defining forms of
		$\mathcal A$. Note that the Jacobian ideal $J_f$ is a codimension $2$ subideal of $I$.
		
%	Recall that if $J\subseteqeq I$ are ideals in a Noetherian ring, then $J$ is a reduction of $I$ if $I^{r+1}=JI^r$ for some sufficiently large $n$. The least such $r$ is the reduction number of $I$ with respect to $J$, denoted $\operatorname{red}_J(I)$. A reduction is minimal if it is minimal with respect	to inclusion.
%The first conjecture of Burity--Simis--Toh\u{a}neanu asserts that $J_f$ is a minimal reduction of $I$. They proved this conjecture in several cases, in particular for $(n-1)$-generic arrangements. We prove it for arbitrary central arrangements in characteristic zero by using the Orlik--Terao presentation of the special fiber of $I$.
		
Let $\mathfrak{R}=k[T_1,\ldots,T_m]$ stand for a polynomial ring over $k$ in new indeterminates. Recall from \cite{OrlikTerao} and \cite[Theorem~2.4]{GST} the Orlik--Terao ideal of $\mathcal A$ as the kernel $Q_{\mathcal A}$ of the $k$-algebra surjection
		\[
		\mathfrak{R}\longrightarrow k\left[\frac{1}{\ell_1},\ldots,\frac{1}{\ell_m}\right],
		\quad
		T_i\longmapsto \frac{1}{\ell_i}.
		\]
		Since $L_i=f/\ell_i$, the homogeneous relations among the reciprocals
		$1/\ell_i$ are the same as the homogeneous relations among the products
		$L_i$. 
	%Hence
	%	\[
	%	Q_{\mathcal A}=
	%	\ker\left(
	%	\mathfrak{R}\longrightarrow k[L_1,\ldots,L_m],
	%	\qquad
	%	T_i\longmapsto L_i
	%	\right).
	%	\]
	It follows that the Orlik--Terao algebra is the special fiber of $I$: 
		\[
		\mathcal F(I)
		%\bigoplus_{q\geq0}I^q/\mathfrak m I^q
		\simeq k[L_1,\ldots,L_m]
		\simeq \mathfrak{R}/Q_{\mathcal A}
	%	,\qquad \mathfrak m=(x_1,\ldots,x_n).
		\]
		
Next recall that a nonempty subset $C\subseteq\{1,\ldots,m\}$ is called a circuit of $\mathcal A$ provided $\{\ell_i\}_{i\in C}$ is linearly dependent and every proper subset
		$\{\ell_i\}_{i\in D}$, $D\subsetneq C$, is linearly independent. 
%Equivalently, $\codim\bigcap_{i\in C}H_i<|C|$, while $\codim\bigcap_{i\in D}H_i=|D|$ for every proper subset $D\subsetneq C$.
		Then, $Q_{\mathcal A}$ is generated by the circuit relations in the sense that,  if $C\subseteq\{1,\ldots,m\}$ is a circuit and
	$\sum_{i\in C}c_i\ell_i=0,\, c_i\in k^*,$
		then
		\[
		\partial C:=
		\sum_{i\in C}c_i
		\prod_{\substack{h\in C\\ h\neq i}}T_h \in Q_{\mathcal A}
		\]
		generate $Q_{\mathcal A}$ as $C$ runs over all circuits of $\mathcal A$.
		
%For $i=1,\ldots,m$, write $\ell_i=\sum_{j=1}^n a_{ji}x_j$. Then the image of $f_{x_j}$ for $j=1,\ldots,n$ in $\mathcal F(I)$ is the linear form 

	We rephrase Conjecture~\ref{red_conj} as a theorem.	
	\begin{Theorem}[$k$ of characteristic zero] \label{thm:BST-reduction-conjecture}
	 The gradient ideal $J_f$ is a minimal reduction of $I$ and
			$\operatorname{red}_{J_f}(I)\leq n-1$.
		\end{Theorem}
		
		\begin{proof}
			We first prove the result when $k$ is algebraically closed.
			
			Let $a_{j,i}: = \partial l_i /\partial x_j$ be the $x_j$-coefficient of $l_i$, and set 
			$\eta_j:=\sum_{i=1}^m a_{ji}T_i$.

%The inclusion $J_f\subseteq I$ follows from $f_{x_j}=\sum_{i=1}^m a_{ji}L_i$. 
			By the fiber criterion of
			\cite[Lemma~1.2]{BST-almost}, it is enough to prove that the ideal
			\[
			(Q_{\mathcal A},\eta_1,\ldots,\eta_n)\subseteq \mathfrak{R}
			\]
			is primary to the maximal ideal $(T_1,\ldots,T_m)$.
			
			Suppose that $0\neq\lambda=(\lambda_1,\ldots,\lambda_m)\in k^m$ is a
			zero of $V(Q_{\mathcal A},\eta_1,\ldots,\eta_n)\subseteq \AA^{m}$. Set
			$S=\{i\mid \lambda_i\neq0\}$. Since $\eta_j(\lambda)=0$ for every $1\leq j\leq n$, one has
			\[
			\sum_{i\in S}\lambda_i\ell_i=0,
			\]
			 a nonzero $k$-linear dependence relation triggering the following element of the
			 Orlik--Terao ideal:
			\[
			G_S:=\sum_{i\in S}\lambda_i
			\prod_{\substack{h\in S\\ h\neq i}}T_h\in Q_{\mathcal A}.
			\]
			Now, since $\lambda\in V(Q_{\mathcal A})$ then $G_S(\lambda)=0$. However,
			\[
			G_S(\lambda)
			=
			\sum_{i\in S}\lambda_i
			\prod_{\substack{h\in S\\ h\neq i}}\lambda_h
			=
			|S|\prod_{h\in S}\lambda_h\neq0,
			\]
			because $\operatorname{char}k=0$ and every $\lambda_h\neq 0$, for every $h\in S$. This contradiction shows that the only zero of
			$V(Q_{\mathcal A},\eta_1,\ldots,\eta_n)$ is the origin. Hence
			\[
			\sqrt{(Q_{\mathcal A},\eta_1,\ldots,\eta_n)}=(T_1,\ldots,T_m),
			\]
			and therefore $(Q_{\mathcal A},\eta_1,\ldots,\eta_n)$ is
			$(T_1,\ldots,T_m)$-primary. Thus $J_f$ is a reduction of $I$.
			
			Since $\mathcal F(I)$ is the Orlik--Terao algebra, one has
			$\ell(I)=\dim\mathcal F(I)=n$; see \cite[Corollary~2.6]{GST}. Since $J_f$ is
			generated by $n$ elements, this reduction is minimal.
			
			Finally, by \cite[Corollary~2.6(d)]{GST}, $\mathcal F(I)$ is Cohen--Macaulay and
			 the (absolute) reduction number is bounded by $n-1$.
			Now, 	$f_{x_j} = \sum_{i=1}^m a_{ji}\frac{f}{\ell_i}
			= \sum_{i=1}^m a_{ji}L_i$, for $1\leq j\leq n$, hence the respective preimages -- $\{\eta_j, 1\leq j\leq n\}$ modulo $Q_{\mathcal A}$ -- in $\mathcal F(I)$
			 form a linear system of parameters
			of $\mathcal F(I)$. Therefore,
			$
			\operatorname{red}_{J_f}(I)\leq n-1,
			$
			as well.
			
			This settles the proof for $k$ algebraically closed.
			To conclude, in general let $\overline{k}$ denote the
			algebraic closure of $k$. Set $\overline{R}=R\otimes_k\overline{k}\simeq \overline{k}[x_1,\ldots,x_n]$, $\overline{I}=I\overline{R}$ and $\overline{J}=J_f\overline{R}$. 
			Let $\overline{f}=f\otimes 1\in\overline{R}$. Since differentiation with
			respect to $x_j$ is computed term by term on the variables, it commutes with
			extension of scalars, namely,
			\[
			(\overline{f})_{x_j}=(f_{x_j})\otimes1
			\qquad\text{for }j=1,\ldots,n.
			\]
			Hence $\overline{J}=J_f\overline{R}=J_{\overline{f}}$.
			The induced arrangement over $\overline{k}$ still has rank $n$, so the algebraically closed case gives
			$\overline{I}^{\,n}=\overline{J}\,\overline{I}^{\,n-1}$. Thus
			\[
			\bigl(I^n/J_fI^{n-1}\bigr)\otimes_R\overline R
			\simeq
			I^n\overline R/J_fI^{n-1}\overline R=\overline{I}^{\,n}/\overline{J}\,\overline{I}^{\,n-1}
			=0.
			\]
			Since $\overline R$ is faithfully flat over $R$, it follows that $I^n/J_fI^{n-1}=0$, hence $I^n=J_fI^{n-1}$. Thus $J_f$ is a reduction of $I$ and
			$\operatorname{red}_{J_f}(I)\leq n-1$ over $k$. 
			
			It remains to descend minimality. Suppose $\mathcal{J}\subseteq J_f$ is a reduction of
			$I$. Then, $\mathcal{J}\overline R\subseteq \overline J$ is a reduction of $\overline I$. Since $\overline J$
			is a minimal reduction of $\overline I$, one has $\mathcal{J}\overline R=\overline J$. Hence
			\[
			(J_f/\mathcal{J})\otimes_R\overline R=0.
			\]
			By faithful flatness, $J_f/\mathcal{J}=0$, so $\mathcal{J}=J_f$. Therefore $J_f$ is a minimal
			reduction of $I$ over $k$.
		\end{proof}
		
		\begin{Corollary}\label{cor:OT-artinian-reduction}
			With the previous notation, one has:
			\begin{enumerate}
				\item[{\rm (i)}] The ring\,
		$\mathcal G_{\mathcal A}:=
	\mathfrak{R}/(Q_{\mathcal A}, \eta_1,\ldots,\eta_n)$
			is an Artinian reduction of the Orlik--Terao algebra. Moreover,
			$(\mathcal G_{\mathcal A})_q=0$, for $q\geq n.$
			\item[{\rm (ii)}] The embedding dimension of the above Artinian Orlik--Terao reduction is
			\[
			\operatorname{embdim}\mathcal G_{\mathcal A}=m-n.
			\]
			Equivalently, the excess difference $m-n$ is the number of independent $k$-linear parameters
			remaining after the Jacobian reduction.
		\end{enumerate}
		\end{Corollary}
		
		\begin{proof}
			 (i) By Theorem~\ref{thm:BST-reduction-conjecture} and its proof,  $\eta_1,\ldots,\eta_n$ give a linear system of parameters over $\mathcal F(I)$. The
			vanishing in degrees $q\geq n$ follows from
			$\operatorname{red}_{J_f}(I)\leq n-1$.

	       (ii)
			The coefficient matrix $(a_{ji})$ of the arrangement has rank $n$, because so does the arrangement. Therefore
			$\mathfrak{R}/(\eta_1,\ldots,\eta_n)$
			is a polynomial ring in $m-n$ variables. Since the Orlik--Terao ideal
			$Q_{\mathcal A}$ has no linear forms, the embedding dimension of
			$\mathfrak{R}/(Q_{\mathcal A},\eta_1,\ldots,\eta_n)$
			is exactly $m-n$.
		\end{proof}

		%%%%%%%%%%%%%%%%%%%%%%%%%%%%%%%%%%%%%%%%%%%%%%%%
		\section{The linear type conjecture}
		One collects a few well-known preliminaries.
		
		Quite generally, let $A$ be a Noetherian ring and let $J=(f_1,\ldots,f_m)\subseteq A$ be an ideal. The Rees
		algebra of $J$ is the $A$-algebra
		\[
		\Rees_A(J):=\bigoplus_{q\geq 0}J^q \simeq A[Jt]=A[f_1t, \ldots, f_mt]\subseteq A[t], \; t\; \text{a variable}.
		\]
		Consider the graded polynomial presentation
		\[
		S:=A[T_1,\ldots,T_m]\lar \Rees_A(J),\qquad T_i\longmapsto f_it,
		\]
		where $S$ is attributed the standard $T$-grading, with $\Q$ standing for its kernel, a homogeneous ideal in the $T$-grading. Denote by
		$\Q_1\subseteq \Q$ the subideal generated by the elements of $\Q$ of $T$-degree one.
	%	Then
	%	\[
	%	\sym_A(J)\simeq S/\Q_1,
	%	\]
	%	and the natural map $\sym_A(J)\surjects \Rees_A(J)$ has kernel $\Q/\Q_1$.
	%	Thus $J$ is of linear type if and only if $\Q=\Q_1$. 
	
	At the other end, let
		\[
		A^q\stackrel{\Phi}{\lar}A^m\lar J\lar 0
		\]
		be a free presentation of $J$ based on the stated generators, and set ${\bf T}:=[T_1\ \cdots\ T_m]$. Then, the defining ideal of a corresponding polynomial presentation of the symmetric algebra $\sym_A(J)$ is
		\[
		\mathcal L:={\bf T}\Phi=I_1({\bf T}\Phi)=\Q_1\subseteq S.
		\]
		Recall that $J$ is said to be an ideal of linear type if the  surjection $\sym_A(J)\surjects \Rees_A(J)$ that maps the symmetric power $\sym_A^q(J)$ to $J^q$ is also injective. 
		Thus, $J$ is of linear type if and only if  $\Q=\mathcal L$, for (any) choice of generators of $J$ as above.
		
		\begin{Remark}\label{saturation_obs}
		Suppose moreover that $A$ is local with maximal ideal $\fm$, and that  the  map
		$\sym_A(J)\to\Rees_A(J)$ is an isomorphism locally on	${\rm Spec}(A)\setminus\{\fm\}$. Since $\Rees_A(J)\subseteq A[t]$ is $A$-torsion-free, one gets
		\[
		\Q/\mathcal L=H^0_{\fm}(S/\mathcal L)
		=(\mathcal L:\fm^\infty)/\mathcal L,
		\]
		i.e., $\Q=\mathcal L:\fm^\infty$.
		
		This formulation will be used frequently below.
	\end{Remark}
		
		The next proposition will show that, for a four-generated ideal, a resolution of type $1,4,4,1$, along with the assumption of ``punctured'' linear type, force a complete Pfaffian presentation of the defining ideal of the Rees algebra.
		
		Parts of this approach have been established in \cite[Section 3.5]{Vasc} and \cite{JacDual1993}.
		
		\begin{Proposition}\label{thm:pfaffian-from-resolution}
			Let $R:=k[x,y,z,w]$ stand for a standard graded polynomial ring over a field $k$, and let
			$J\subseteq R$ denote an ideal generated by four homogeneous polynomials of the same degree.
			Let $\Q\subseteq S:=R[T_1,T_2,T_3,T_4]$ 
			%be the defining ideal of $\Rees_R(J)$, 
			and $\mathcal L=\Q_1$ be as above. Assume that$:$
			
		\begin{enumerate}
			\item[{\rm (a)}] Upon a choice of equigenerated generators of $J$ and base change of $R$, $R/J$ has a minimal graded free resolution of the form
				\[
				0\to R \xrightarrow{u} R^4 \xrightarrow{\Phi} R^4 \to R
				\to R/J\to 0,
				\]
				where $u:=(x,y,z,w)^t;$
				
			\item[{\rm (b)}] $J$ is locally of linear type on ${\rm Spec}(R)\setminus \{(x,y,z,w)\}$.
			\end{enumerate}
			Then, there exists a $4\times4$ alternating matrix $B$ over $S$ such that $\Q=(\mathcal{L},P)$, where 
			%$T\Phi=[x\ y\ z\ w]B$.
			$P:=\operatorname{pf}(B)$ denotes the Pfaffian of $B$. 
			%then
			%\[	\Q=\mathcal L+(P).	\]
			In particular, $J$ is not of linear type.
		\end{Proposition}
		\begin{proof} 
			
		{\bf Claim.} 
		With the assumption of item (a), there exists a $4\times4$ alternating matrix $B$ over $S$ such that ${\bf T}\Phi=[x\ y\ z\ w]B$, and, moreover, the entries of $B$ belong to the maximal ideal $\fn:=(x,y,z,w,T_1,T_2,T_3,T_4)$ of $S$.
		
		For this, one may assume that $R=k[x,y,z,w]_{(x,y,z,w)}$.
		The main assertion is proved in \cite[Proposition 2.4]{JacDual1993}.
		For the additional assertion, otherwise some entry of the matrix $[x\ y\ z\ w]B$ would have at least one term of the form $ay$, with $a\in k\setminus \{0\}$ and $y\in\{x,y,z,w\}$.
			However, no such term can be a term of an entry of the matrix ${\bf T}\Phi$ because the entries of $\Phi$ are homogeneous of degree $\geq 1$.
		
		So much for the above claim.

			Now, let $\mathrm{adj}(B)$ be the adjugate of $B$, so that $B\mathrm{adj}(B)=\operatorname{pf}(B)\operatorname{Id}_4=P\operatorname{Id}_4$. Since $[x\ y\ z\ w]B=0$ in
			$S/\mathcal L$, multiplying by $\mathrm{adj}(B)$ gives
			$P[x\ y\ z\ w]=0$  in  $S/\mathcal L$.
			Thus $\fm P\subseteq\mathcal L$. 
		%	Set $K=\mathcal L+(P)$. 
			Then
			\[K:=
			(\mathcal{L},P)\subseteq \mathcal L:\fm\subseteq \mathcal L:\fm^\infty=\Q,
			\]
			where the last most equality follow by item (b) and Remark~\ref{saturation_obs}.
			
			We now prove the reverse inclusion, for which one may assume that $S$ is local with unique maximal ideal $\fn$. Consider the $5\times5$ alternating matrix
			\begin{equation}\label{skew-symmetric}
			M=
			\begin{pmatrix}
				0&x&y&z&w\\
				-x&0&a&b&c\\
				-y&-a&0&d&e\\
				-z&-b&-d&0&h\\
				-w&-c&-e&-h&0
			\end{pmatrix},\; (a,b,c,d,e,h\in S),
			\end{equation}
			of which $B$ is assumed to be the submatrix by removing the first column and row.
			Thus, its submaximal Pfaffians are, up to entry interchange in $B$ and signs, the four entries of
			$[x\ y\ z\ w]B$ along with $P$. Hence, $K$ is
			the ideal of submaximal Pfaffians of $M$.

		{\bf Claim.} $\hht K=3$. 
		
		Since $K\subseteq\Q$, then
		$\hht K \leq\hht\Q= 3$. Let $\fp$ be a minimal prime of $K$. If $\fp\supseteq\fm S$, then $\hht\fp\geq4$. If $\fp\not\supseteq\fm S$, then localizing at $\fp$ gives $\Q_\fp=\mathcal L_\fp$, because $\Q/\mathcal L$ is supported at $\fm$.
			Since $K_\fp\subseteq\Q_\fp$ and $K_\fp$ is proper, it follows that
			$\Q\subseteq\fp$, and therefore $\hht\fp\geq3$. 
	Thus every minimal prime of	$K$ has height at least $3$, and so $\hht K=3$.
			
			As observed above, the entries of $M$ belong to the maximal ideal of $S$.
			Then, by \cite[Theorem~2.1]{BE77}, one knows that the ideal of submaximal Pfaffians of $M$ is a perfect Gorenstein ideal of height three, and -- since the entries of $M$ lie in $\fn \subseteq S$ --  form a minimal generating set. In particular,
			$P\notin\mathcal L$.
			
			Moreover, $S/K$ is Cohen--Macaulay and all its associated primes have height
			$3$. No associated prime of $S/K$ contains $\fm S$, because every prime
			containing $\fm S$ has height at least $4$. Hence
			$H^0_{\fm}(S/K)=0$, equivalently $K:\fm^\infty=K$.
			Now $\mathcal L\subseteq K$, so
			$\mathcal L:\fm^\infty\subseteq K:\fm^\infty=K$. Therefore,
			\[
			\Q=\mathcal L:\fm^\infty\subseteq K,
			\]
		as was to be shown.
		\end{proof}

		The intent is to apply the above proposition in the case where $J$ is the gradient ideal of a form in $R=k[x,y,z,w]$.
		
		The assumption of item (a) is quite particular. Perhaps not so much the assumption of item (b), which one now envisages in the case of the gradient of the defining form of a central arrangement.

		\begin{Proposition}[char$(k)=0$]
			\label{prop:rank-four-Rees-saturation}
			 Let
			$f=\ell_1\cdots \ell_m\in R=k[x,y,z,w]$ be the defining polynomial of a central arrangement. Then the gradient ideal $J_f$ is locally of linear type on ${\rm Spec}(R)\setminus \{(x,y,z,w)\}$. 
		\end{Proposition}
		
		\begin{proof}
			The gradient ideal of a central arrangement of rank at most $3$ is of linear
			type: ranks $1$ and $2$ are elementary, and rank $3$ follows from
			\cite[Proposition~2.14]{BST-almost}.
			Thus, it suffices to prove that, locally on the ``punctured spectrum'', $J_f$ coincides with the gradient ideal of an arrangement of rank at most $3$.
			
			To proceed, let $\fp\neq\fm$ be any non-maximal prime. Set
			\[
			g_\fp=\prod_{\ell_i\in\fp}\ell_i,
			\]
			with the convention that $g_\fp=1$ if no $\ell_i$ belongs to $\fp$. Write
			$f=ug_\fp$, where $u=\prod_{\ell_i\notin\fp}\ell_i$ is a unit in $R_\fp$.
			If $g_\fp=1$, then $f$ is a unit in $R_\fp$. Since $f\in J_f$, we get
			$(J_f)_\fp=R_\fp$, hence $(J_f)_\fp$ is of linear type. Assume now that
			$g_\fp\neq 1$. For every $v\in\{x,y,z,w\}$,
			\[
			f_v=u(g_\fp)_v+g_\fp u_v.
			\]
			Euler's formula gives $g_\fp\in J_{g_\fp}$, hence
			$(J_f)_\fp\subseteq (J_{g_\fp})_\fp$. Conversely, since $f\in J_f$ and
			$f=ug_\fp$ with $u$ a unit, we have $g_\fp\in (J_f)_\fp$. Thus
			\[
			u(g_\fp)_v=f_v-g_\fp u_v\in (J_f)_\fp,
			\]
			and therefore $(g_\fp)_v\in (J_f)_\fp$. Hence
			\[
			(J_f)_\fp=(J_{g_\fp})_\fp.
			\]
			Since $\fp\neq\fm$, the linear forms contained in $\fp$ span a proper subspace
			of $R_1$. Thus $g_\fp$ defines a central arrangement of rank at most $3$, as required.
	\end{proof}	
%%%%%%%%%%%%%%%%%%%%%%%%%%%%%%%%%%%%			
%			Now
%			\[
%			\Q/\mathcal L =\ker\big(\sym_R(J_f) \longrightarrow \Rees_R(J_f)\big).
%			\]
%			By the first part, this kernel is supported at $\fm$. Moreover, away from $\fm$
%			we have
%			\[
%			\sym_R(J_f)_\fp\simeq \Rees_R(J_f)_\fp\subseteq R_\fp[t],
%			\]
%			so $\sym_R(J_f)$ is $R$-torsion-free locally on
%			$\spec{R}\setminus\{\fm\}$. By Simis's torsion criterion
%			\cite[Lemma~5.2]{Simis-ELAM}, the $R$-torsion of $\sym_R(J_f)$ is $0:_{\sym_R(J_f)}\fm^\infty$. Since $\Rees_R(J_f)\subseteq R[t]$ is $R$-torsion-free, the kernel
%			$\Q/\mathcal L$ is exactly this torsion submodule. Therefore
%			\[
%			\Q/\mathcal L
%			=
%			0:_{\sym_R(J_f)}\fm^\infty
%			=
%			(\mathcal L:\fm^\infty)/\mathcal L,
%			\]
%			and hence $\Q=\mathcal L:\fm^\infty$. 
%%%%%%%%%%%%%%%%%%%%%%%%%%%%%%%%%%%%
		
One next gives an example of a central arrangement of maximal rank in $R=k[x,y,z,w]$ satisfying the assumptions of Proposition~\ref{thm:pfaffian-from-resolution} with $J=J_f$ the gradient ideal of the defining polynomial $f$ of the arrangement.

\begin{Example}\label{main-ex}
$f:=xyzw(w+y)(w+z)(w-x+y)(w-x+z)(w-x+y+z).$
\end{Example}
By Proposition~\ref{prop:rank-four-Rees-saturation}, it suffices to prove the assumption of item (a).
For this one resorts to a computation in~\cite{Singular} yielding that $R/J_f$ has a minimal graded free resolution of
the form
\[
0\to R(-12)\xrightarrow{u}R(-11)^4\xrightarrow{\Phi}R(-8)^4
\xrightarrow{[f_x\ f_y\ f_z\ f_w]}R\to R/J_f\to 0,
\]
where, upon choice of bases, one may assume that $u=(x,y,z,w)^t$.

Note that the entries of $\Phi$ are forms of degree $3$ in $R$, hence every nonzero entry of the skew-symmetric matrix $B$ has bidegree $(2,1)$ in $S$.
Therefore, with the notation of $B$ in (\ref{skew-symmetric}), its Pfaffian $P=(ah-be+cd)$ has bidegree $(2,1)+(2,1)=(4,2)$ in $S$.
In particular, it is an obstruction to $J_f$ being syzygetic in the terminology of \cite{SV-conormal}. 
In the subsequent section one will expand on the combinatorial details of this example.

		The above example persists upon multiplying by Boolean factors in
		new variables.
		One presents this as a theorem stating the full counter-part disproving Conjecture~\ref{conj:BST}.
		
		\begin{Theorem}\label{thm:counterexamples-all-ranks}
			For every $n\geq 4$, there exists   a central hyperplane arrangement in the affine space $\AA_k^n$ whose gradient ideal is not of linear type.
		\end{Theorem}
		\begin{proof}
			The case $n=4$ is Example~\ref{main-ex}. Let
			$n=4+s$ with $s\geq 1$, and set
			\[
			R_s=k[x,y,z,w,u_1,\ldots,u_s],
			\qquad
			f_s=fu_1\cdots u_s,
			\]
			where $f$ is as in the above example.
			Then $f_s$ defines a central arrangement of rank $4+s$. Let
			$\fp=(x,y,z,w)R_s\subseteq R_s$. Since no $u_i$ belongs to $\fp$, it must be a unit in $(R_s)_\fp$. Hence $(J_{f_s})_\fp=J_f(R_s)_\fp$. 
			Indeed, for the old variables $v\in\{x,y,z,w\}$ one has
			\[
			(f_s)_v=(u_1\cdots u_s)f_v,
			\]
			while for the new ones,
			\[
			(f_s)_{u_i}=f u_1\cdots \widehat{u_i}\cdots u_s.
			\]
			Since $f$ is homogeneous, Euler's formula gives $f\in J_f$, and the equality
			follows.
			Assume, by contradiction, that $J_{f_s}$ is of linear type. Then $(J_{f_s})_\fp$ is of linear type. Hence $J_f(R_s)_\fp$ is of linear type.
			But the map
			\[
			k[x,y,z,w]_{(x,y,z,w)}\lar (R_s)_\fp
			\]
			is faithfully flat. Since the linear type condition descends by faithfully flat base
			change, $J_f$ would be of linear type at $(x,y,z,w)$, a contradiction. 
			Therefore $J_{f_s}$ is not of
			linear type either.
		\end{proof}
		
		%%%%%%%%%%%%%%%%%%%%%%%%%%%%%%%%%%%%%%%%%%%%%%%
		%%%%%%%%%%%%%%%%%%%%%%%%%%%%%%%%%%%%%%%%%%	
		\section{Around the combinatorics of the Pfaffian obstruction}
		\label{sec:combinatorics-pfaffian-obstruction}
		
		One now explains how the rank-four counterexample is reflected in the
		intersection lattice of the arrangement. The point being that the Pfaffian
		equation of Example~\ref{main-ex} is not an isolated
		homological accident: it appears in the presence of two combinatorial
		phenomena, namely a nontrivial space of linear relations among the defining
		forms and nonordinary codimension two flats.
		
		One briefly recall some intervening combinatorial notions, for which the overall reference is \cite{OrlikTerao}.
		
		Let $\mathcal A=\{H_i={\mathcal{V}}(\ell_i)\}_{i=1}^m\subseteq k^n$ be a central arrangement of rank
		$n$ over $k$, with defining polynomial $f:=\ell_1\cdots\ell_m \in R:=k[x_1,\ldots,x_n]$. Let
		$L(\mathcal A)$ be the  lattice of intersections of $H_i$'s, ordered by reverse inclusion, an element of which is called a \emph{flat}. 
		Given a flat
		$X\in L(\mathcal A)$, one sets $\rk(X):=\codim(X)$, as a $k$-vector subspace.
		Of particular interest here is the sub-lattice
		\[
		L_2(\mathcal A)=\{X\in L(\mathcal A)\mid \rk(X)=2\}.
		\]
		For $X\in L_2(\mathcal A)$,  denote
		\[
		S_X:=\{1\leq i\leq m\mid X\subseteq H_i\},\qquad m_X:=|S_X|,
		\qquad P_X:=(\ell_i\mid i\in S_X).
		\]
		Thus, $m_X$ is the number of hyperplanes of $\mathcal A$ passing through the codimension two
		flat $X$, while $P_X$ is the prime ideal generated by the corresponding linear forms. Introduce the following additional definitions:
		\[
		\epsilon(\mathcal A):=m-n,\qquad
		\mu_2(\mathcal A):=\max_{X\in L_2(\mathcal A)}m_X,\qquad
		\nu_3(\mathcal A):=\sum_{X\in L_2(\mathcal A)}\binom{m_X}{3}.
		\]
		The bearing of these elements is as follows: $\epsilon(\mathcal A)$ is the $k$-vector dimension of the kernel of the map defined by the $n\times m$ coefficient matrix
		$(a_{ji})$ of the defining forms
		$\ell_i=\sum_j a_{ji}x_j$, i.e., 
		the $k$-vector dimension of the space of linear relations
		among the defining forms; while $\mu_2(\mathcal A)$ and $\nu_3(\mathcal A)$
		measure each the codimension two ``fatness'' of the arrangement. 
		
		The notion of a circuit of $\mathcal A$ has been recorded earlier. 
		%is a minimally $k$-linearly %\emph{dependent} subset $C:=\{l_{i_1}, %\ldots, \ell_{i_c}\}$ of $\{\ell_1,\ldots,\ell_m\}$, i.e., every %proper subset of $C$ is $k$-linearly %\emph{independent}. 
		One denotes by $g(\mathcal A)$  the smallest cardinality of a circuit. 
		%Equivalently, $g(\mathcal A)$ is the smallest number of defining forms of $\mathcal A$ which are linearly dependent. 
		If there is no circuit, one sets $g(\mathcal A)=\infty$.
		
		As in previous sections, let $L_i:=f/\ell_i$, for $1\leq i\leq m$, and 
		$ I:=(L_1,\ldots,L_m)$, the corresponding ideal of $(m-1)$-fold products. 
		
		The following proposition shows that $I$ is essentially controlled in terms of codimension two flat data:
		
		\begin{Proposition}\label{prop:reciprocal-fat-flats}
			With the above notation, one has$:$
			
		%	\begin{itemize}
			%	\item[{\rm (1)}] 
			$\bullet$ \emph{Decomposition}$:$
				$
				I=\bigcap_{X\in L_2(\mathcal A)}P_X^{m_X-1}.
				$
			%	\item[{\rm (2)}] 
			
			$\bullet$	\emph{Degree}$:$
				\[
				\deg(R/I)=\sum_{X\in L_2(\mathcal A)}\binom{m_X}{2}=\binom{m}{2}.
				\]
		%	\item[{\rm (3)}]
			
			$\bullet$ The following conditions are equivalent:
	%	\end{itemize}
			\begin{itemize}
				 \item[{\rm (a)}] $I$ is generically a complete intersection.
				\item[{\rm (b)}] $m_X=2$ for every $X\in L_2(\mathcal A).$
				\item[{\rm (c)}] $\mu_2(\mathcal A)=2.$
				\item[{\rm (d)}] $g(\mathcal A)\geq 4$.
			\end{itemize}
		\end{Proposition}
		
		\begin{proof}
			The decomposition is the standard decomposition of the ideal of
			$(m-1)$-fold products -- see~\cite[Proposition~2.2]{AGT}. One recasts the argument for the reader's convenience.
			
			 Let $X\in L_2(\mathcal A)$. Since $\codim X=2$, the ideal
			\[
			P_X=(\ell_i\mid X\subseteq H_i)
			\]
			is the defining prime ideal of the $k$-subspace $X$, in particular $\hht(P_X)=2$.
			Upon localizing at $P_X$, the forms $\ell_i$ with $i\notin S_X$ become units.
			Therefore,
			\[
			I_{P_X}=
			\left(
			\prod_{\substack{j\in S_X\\ j\neq i}}\ell_j
			\ \middle|\ i\in S_X
			\right)R_{P_X}.
			\]
			Since $R_{P_X}$ is regular local ring of dimension $2$, with maximal ideal
			$P_XR_{P_X}$, it follows that
		%	. If $m_X=r$, then the forms $\ell_i$, $i\in S_X$, are distinct
		%	linear forms spanning $P_XR_{P_X}$.
		 the ideal generated by the products
			of one less defining form is $P_X^{m_X-1}R_{P_X}$.
			Thus $I_{P_X}=P_X^{m_X-1}R_{P_X}$.
			Since $I$ is unmixed of height $2$, the claimed decomposition follows.
			
		To compute the degree, since $I$ is unmixed of height $2$, the associativity formula for multiplicities gives
			\[
			\deg(R/I)=
			\sum_{X\in L_2(\mathcal A)}
			\lambda (R_{P_X}/I_{P_X})\,\deg(R/P_X).
			\]
			Clearly $\deg(R/P_X)=1$, as $P_X$ is generated by $k$-linear forms.
			Moreover,
			\[
			(R/I)_{P_X}\simeq R_{P_X}/P_X^{m_X-1}R_{P_X}.
			\]
			If $\mathfrak n$ is the maximal ideal of a regular local ring $\mathfrak N$ of dimension $2$
			then, for any $s\geq 1$,
			\[
			\lambda(\mathfrak N/\mathfrak n^s)=1+2+\cdots+s=\binom{s+1}{2}.
			\]
			With $\mathfrak N=R_{P_X}$ and $s=m_X-1$ it gives
			\[
			\lambda \bigl(R_{P_X}/P_X^{m_X-1}R_{P_X}\bigr)
			=
			\binom{m_X}{2}.
			\]
			Therefore
			\[
			\deg(R/I)=\sum_{X\in L_2(\mathcal A)}\binom{m_X}{2}.
			\]
			This establishes the first equality.
			For the second one, note that any pair of hyperplanes determines a unique codimension two flat,
			and any flat $X$ containing $m_X$ hyperplanes accounts for exactly
			$\binom{m_X}{2}$ such pairs. Hence
			\[
			\sum_{X\in L_2(\mathcal A)}\binom{m_X}{2}=\binom{m}{2},
			\]
			as claimed.

			It remains to prove the stated equivalences. By the decomposition, $I$ is generically
			a complete intersection if and only if $P_X^{m_X-1}R_{P_X}$ is a complete
			intersection for every $X\in L_2(\mathcal A)$. In a two-dimensional regular
			local ring, the power $\mathfrak n^s$ is a complete intersection if and only if
			$s=1$. Thus this happens if and only if $m_X-1=1$, that is, $m_X=2$, for every
			$X$. This is equivalent to having $\mu_2(\mathcal A)=2$.
			
			Finally, if $m_X\geq3$ for some codimension two flat $X$, then at least three
			defining forms become $k$-linearly dependent on the same two-dimensional span.
			Equivalently, $\mathcal A$ has a $3$-circuit. Hence, the assumption that $m_X=2$ for every
			$X\in L_2(\mathcal A)$ is equivalent to the absence of $3$-circuits, i.e. to having
			$g(\mathcal A)\geq4$.
			\end{proof}

		In Corollary~\ref{cor:OT-artinian-reduction} one introduced the Artinian Orlik--Terao reduction
		\[
		\mathcal G_{\mathcal A}
		=
		\mathcal F(I)/J_f\mathcal F(I)
		\simeq
		\mathfrak{R}/(Q_{\mathcal A},\eta_1,\ldots,\eta_n),
		\]
		where $\mathfrak{R}=k[T]:=k[T_1,\ldots,T_m]$, $Q_{\mathcal A}$ is the defining ideal of the Orlik--Terao algebra $\mathcal F(I)$, and
		$\eta_j=\sum_{i=1}^m a_{ji}T_i\, (1\leq j\leq n)$ its ${T}$-linear part. It came along with the assertion that its embedding dimension is $m-n=\epsilon(\mathcal A)$.
		One next explain how this intertwines with the circuit length bound $g(\mathcal{A})$.
		
		As usual, $\mathfrak{A}_d$ stands for the $k$-span of degree $d$ of a graded $k$-algebra $\mathfrak{A}$.
%By Corollary~\ref{cor:excess-embedding-dimension},  one has
%		\[\operatorname{embdim}\mathcal G_{\mathcal A}=m-n=\epsilon(\mathcal A).	\]
	%	Thus the number $m-n$ is not merely numerical: it is the number of independent
	%	linear parameters remaining after the Jacobian reduction.
		
		\begin{Proposition}\label{prop:OT-reduction-low-degrees}
			With the above notation and convention, one has$:$
			\begin{enumerate}
				\item[{\rm(i)}]
			For every $d\leq g(\mathcal A)-2$,
			\[
			\dim_k(\mathcal G_{\mathcal A})_d
			=
			\binom{\epsilon(\mathcal A)+d-1}{d}.
			\]
			\item[{\rm (ii)}] 
			If $q_2(\mathcal A)$ is the number of $k$-independent quadratic equations in the defining ideal of ${\mathcal G}_{\mathcal A}$, then
			\[
			\dim_k(\mathcal G_{\mathcal A})_2
			=
			\binom{m-n+1}{2}-q_2(\mathcal A),
			\qquad
			0\leq q_2(\mathcal A)\leq \nu_3(\mathcal A).
			\]
			\item[{\rm (iii)}]
			If $g(\mathcal A)\geq 4$, then
			\[
			\dim_k(\mathcal G_{\mathcal A})_2=\binom{m-n+1}{2}.
			\]
			In particular, there are no quadratic Orlik--Terao relations.
		\end{enumerate}
		\end{Proposition}
		
		\begin{proof} (i)
			Write $\ell_i=\sum_{j=1}^n a_{ji}x_j$ and let
			$(a_{ji})$ be the associated $n\times m$ coefficient matrix. Since $\mathcal A$ has
			rank $n$, then so does this matrix. Hence, the linear forms $\eta_j$ are linearly independent in $\mathfrak{R}_1$. Therefore,
			upon a $k$-linear change of coordinates in $\mathfrak{R}$, one may assume that 
			$(\eta_1,\ldots,\eta_n)=(T_1,\ldots,T_n)$.
			Thus, $\mathfrak{R}/(\eta_1,\ldots,\eta_n)$
			is a polynomial ring in $m-n=\epsilon(\mathcal A)$ variables. In particular,
			for every $d\geq 0$,
			\[
			\dim_k\bigl(\mathfrak{R}/(\eta_1,\ldots,\eta_n)\bigr)_d
			=
			\binom{\epsilon(\mathcal A)+d-1}{d}.
			\]
			
			Now
			\[
			\mathcal G_{\mathcal A}
			\simeq
			\mathfrak{R}/(Q_{\mathcal A},\eta_1,\ldots,\eta_n)
			\simeq
			\bigl(\mathfrak{R}/(\eta_1,\ldots,\eta_n)\bigr)/
			\overline{Q_{\mathcal A}},
			\]
			where $\overline{Q_{\mathcal A}}$ denotes the image of $Q_{\mathcal A}$ in
			$\mathfrak{R}/(\eta_1,\ldots,\eta_n)$.
			
			The Orlik--Terao ideal $Q_{\mathcal A}$ is generated by circuit relations. If
			$C$ is a circuit of cardinality $r$, then the corresponding circuit relation
			has degree $r-1$. Hence, 
		%the smallest possible degree of anonzero element of $Q_{\mathcal A}$ is 
			$\indeg(Q_{\mathcal A})=g(\mathcal A)-1$. Thus,
			$(Q_{\mathcal A})_d=0$ for every $d\leq g(\mathcal A)-2$.
			Since $Q_{\mathcal A}$ is homogeneous, 
		%passing modulo the homogeneous linear ideal $(\eta_1,\ldots,\eta_n)$ cannot create elements of smaller degree.
		$(\overline{Q_{\mathcal A}})_d=0$ as well,
			for $d\leq g(\mathcal A)-2$.
			Thus, for every such $d$, one has
			$(\mathcal G_{\mathcal A})_d
			\simeq \bigl(\mathfrak R/(\eta_1,\ldots,\eta_n)\bigr)_d$.
			Consequently,
			\[
			\dim_k(\mathcal G_{\mathcal A})_d
			=
			\binom{\epsilon(\mathcal A)+d-1}{d}.
			\]
		
		(ii) Set $\mathfrak P=\mathfrak R/(\eta_1,\ldots,\eta_n)$ and let $\overline{Q_{\mathcal A}}$ be the
		image of $Q_{\mathcal A}$ in $\mathfrak P$. Then
		$\mathcal G_{\mathcal A}\simeq \mathfrak P/\overline{Q_{\mathcal A}}$. Since ${\mathfrak P}$ is a
		polynomial ring in $m-n$ variables, $\dim_k {\mathfrak P}_2=\binom{m-n+1}{2}$. By the
		definition of $q_2(\mathcal A)$, $\dim_k(\overline{Q_{\mathcal A}})_2=
		q_2(\mathcal A)$. Hence
		\[
		\dim_k(\mathcal G_{\mathcal A})_2
		=
		\dim_k {\mathfrak P}_2-\dim_k(\overline{Q_{\mathcal A}})_2
		=
		\binom{m-n+1}{2}-q_2(\mathcal A).
		\]
		
		The space $(Q_{\mathcal A})_2$ is spanned by the Orlik--Terao relations
		attached to $3$-circuits. The number of such circuits is
		$\nu_3(\mathcal A)=\sum_{X\in L_2(\mathcal A)}\binom{m_X}{3}$. Therefore the
		dimension of the image of $(Q_{\mathcal A})_2$ in ${\mathfrak P}_2$ is at most
		$\nu_3(\mathcal A)$, giving
		$0\leq q_2(\mathcal A)\leq \nu_3(\mathcal A)$.
	%	If $g(\mathcal A)\geq 4$, there are no $3$-circuits, so $(Q_{\mathcal A})_2=0$.	
		
		(iii)	
			It follows immediately from (i) and (ii).
		\end{proof}
		
	Returning to Example~\ref{main-ex}, one wishes to highlight two of its features.
	First, the purely combinatorial side, for which one labels the defining forms as
		\[
		\begin{array}{lll}
			\ell_1=x, & \ell_2=y, & \ell_3=z,\\
			\ell_4=w, & \ell_5=w+y, & \ell_6=w+z,\\
			\ell_7=w-x+y, & \ell_8=w-x+z, & \ell_9=w-x+y+z.
		\end{array}
		\]
		%Then
	%	\[n=4,\qquad m=9,\qquad \epsilon(\mathcal A)=9-4=5.\]
		In terms of the index $i$ of $\ell_i$ in the above labeling, the $3$-circuits of the arrangement are 
		\[
		\{1,5,7\},\quad \{1,6,8\},\quad \{2,4,5\},\quad
		\{2,8,9\},\quad \{3,4,6\},\quad \{3,7,9\}.
		\]
		They generate the quadratic sector of the Orlik--Terao ideal $Q_{\mathcal A}.$
	In addition, $g(\mathcal A)=3$, $\epsilon(\mathcal A)=9-4=5$, $\mu_2(\mathcal A)=3$, and $\nu_3(\mathcal A)=6$.  Upon the Jacobian reduction, the remaining degree-one space has dimension $\epsilon(\mathcal A)=5$, so the ambient quadratic space of the Artinian reduction has dimension $\binom{6}{2}=15$. The six triple flats are precisely the codimension two source of quadratic relations in this space.
		
		\vspace{0.2cm}
		
		For the more algebraic counterpart, one emphasizes the torsion Pfaffian element $P$ as a syzygetic obstruction.
		 Following \cite{SV-conormal}, recall the meaning of the syzygetic defect, a notion already mentioned in the previous section. Namely, if $\mathcal J\subseteq R$
		is an ideal, set
		\[
		\delta(\mathcal J):=\ker\bigl(\operatorname{Sym}_R^2(\mathcal J)\to {\mathcal J}^2\bigr).
		\]
		Then $\mathcal J$ is said to be syzygetic if $\delta(\mathcal J)=0$. On the other hand, an inclusion of ideals $\mathcal J\subseteq I$ induces an $R$-map $\operatorname{Sym}_R(\mathcal J)\to
		\operatorname{Sym}_R(I)$.  Define 
		\[
		\delta(\mathcal J,I):=\ker\bigl(\operatorname{Sym}_R^2(\mathcal J)\to
		\operatorname{Sym}_R^2(I)\bigr).
		\]
		
		One deals with the case where $I=(f/\ell_1,\ldots,f/\ell_m)\subseteq R$ as above and $\mathcal J=J_f$ is the gradient of $I$.
		The next proposition isolates the two possible sources of quadratic torsion. The
		ideal $I=(f/\ell_1,\ldots,f/\ell_m)$ has a syzygetic obstruction controlled by
		the $3$-circuits, while that of $J_f$ is subject to the behavior of the inclusion
		$J_f\subseteq I$.
		More precisely:
		
		\begin{Proposition}\label{thm:no-three-circuits-quadratic-obstruction}
			Let $\mathcal A$ be a central arrangement of rank $4$ over a field of
			characteristic zero. Assume that $g(\mathcal A)\geq4$. Then:
			\begin{enumerate}
				\item[{\rm (i)}]  $I$ is
				syzygetic.
				\item[{\rm (ii)}] $\delta(J_f)=\delta(J_f,I)$.	
				\item[{\rm (iii)}]  $J_f$ is syzygetic if and only if the
				map
				$
				\operatorname{Sym}_R^2(J_f)\longrightarrow \operatorname{Sym}_R^2(I)
				$
				is injective.
			\end{enumerate}
		\end{Proposition}
		
		\begin{proof}
		(i)	Let $S=R[T_1,\ldots,T_m]$, and let $\mathcal L\subseteq \mathcal Q$ as previously denote,
			respectively, the defining ideal of $\mathcal R_R(I)$ and its syzygy-theoretic part. By
			\cite[Theorem~4.2]{GST}, the ideal $I$ is of fiber
			type, that is
			\[
			\mathcal Q=(\mathcal L,Q_{\mathcal A}S),
			\]
			where $Q_{\mathcal A}$ as before denotes the defining ideal of the Orlik--Terao fiber algebra.
		Since $g(\mathcal A)\geq4$, item (ii) of Proposition~\ref{prop:OT-reduction-low-degrees} gives
			$(Q_{\mathcal A})_2=0$. Hence,
			\[
			(\mathcal Q/\mathcal L)_2=0.
			\]
			Since $S/\mathcal L\simeq \operatorname{Sym}_R(I)$ and
			$S/\mathcal Q\simeq \mathcal R_R(I)$, this degree-two quotient is precisely
			$\delta(I)$. Thus $\delta(I)=0$.
			
		(ii) In the commutative diagram
			\[
			\begin{array}{ccc}
				\operatorname{Sym}_R^2(J_f) & \longrightarrow &
				\operatorname{Sym}_R^2(I)\\
				\downarrow && \downarrow\\
				J_f^2 & \longrightarrow & I^2 ,
			\end{array}
			\]
			the bottom map is trivially injective,
			and so is the right vertical map by item {\rm (i)}. Hence the kernel of the left vertical
			map is exactly the kernel of the top horizontal map, that is,
			\[
			\delta(J_f)=\delta(J_f,I).
			\]
			This proves {\rm (ii)}, and {\rm (iii)} follows immediately.
		\end{proof}
		Note that, for a rank $4$ arrangement,  if $g(\mathcal A)\geq5$, then
		$\mathcal A$ is generic, hence $J_f$ is of linear type by
		\cite[Theorem~3.10]{LORS}, in particular $\delta(J_f)=0$. Thus, a remnant query is:

		\begin{Question}\label{quest:almost-generic-syzygetic}
			Let $\mathcal A$ be a rank $4$ central arrangement over a field of
			characteristic zero. Assume that $g(\mathcal A)=4$. Is  $J_f$ syzygetic?
		\end{Question}
		
		Example~\ref{main-ex} shows that the hypothesis is sharp from the
		codimension-two point of view. Here one has
		$g(\mathcal A)=3$, $\mu_2(\mathcal A)=3$ and
		$\nu_3(\mathcal A)=6$, and the unique quadratic torsion class is represented
		by the Pfaffian equation. In this sense, the failure to satisfy the linear type property in this example is governed by the interaction between the two numbers
		\[
		\epsilon(\mathcal A)=m-n=5
		\qquad\text{and}\qquad
		\nu_3(\mathcal A)=6,
		\]
		the first of which measures the available space of relations upon the Jacobian
		reduction, while the second measures the quadratic circuit amount coming from
		triple flats.
		
		\emph{Added in time}: Upon the version of this paper on the ``arXiv'', 
		 U.~Walther kindly pointed out to his earlier paper~\cite{Walther}, in which
			hyperplane arrangements are discussed from the viewpoint of logarithmic forms and $D$-modules. In particular, it is shown in \cite[Corollary~3.23]{Walther}  that, under
			suitable tameness hypotheses, the Jacobian ideal is of linear type. In \cite[Example~5.7]{Walther}  a non-tame arrangement is described for which the
			annihilator of $f^s$ is not generated by order-one operators.
			
			Thus, Walther's work reveals, from the logarithmic and $D$-module side, a closely related phenomenon  to the Rees algebra obstruction in terms of a Pfaffian equation as established here.
			 Though this relationship has not been employed in the present paper, the authors thank Walther for his deep remarks and sugestions for future work on the discussions.
		
		%%%%%%%%%%%%%%%%%%%%%%%%%%%%%%%%%%%%%%%%%

		%%%%%%%%%%%%%%%%%%%%%%%%%%%%%%%%%%%%%%%%%%%%%%
		
	\end{document}